\UseRawInputEncoding%%arxiv
\documentclass[11pt,a4paper]{article}
\usepackage{amsfonts}
\usepackage{amssymb}
\usepackage{mathrsfs}
\usepackage{amsmath}
\usepackage{mathtools}
\usepackage{booktabs}
\usepackage{epsf,epsfig,amsfonts,amsgen,indentfirst}
\usepackage{amsmath,amstext,amsbsy,amsopn,amsthm,bbding,wasysym}
\usepackage{multicol,mathdots}
\usepackage{subfigure}
\usepackage{graphicx}
\allowdisplaybreaks

\newtheorem{theorem}{Theorem}[section]

\newtheorem{lemma}[theorem]{Lemma}
\newtheorem{cor}[theorem]{Corollary}

\newtheorem{Theorem}{Theorem}

\begin{document}
\title
{\LARGE \textbf{Two identities of hook immanantal polynomials of (di)graphs and their applications\thanks{ supported by NSFC (No. 12261071)and NSF of Qinghai Province (No. 2020-ZJ-920).} }}

\author{ Tingzeng Wu$^{a,b}$\thanks{{Corresponding author.\newline
\emph{E-mail address}: mathtzwu@163.com, wwzhang863@163.com, hjlai@math.wvu.edu
}}, Wenwei Zhang$^a$, Hong-Jian Lai$^{c,d}$\\
{\small $^{a}$ School of Mathematics and Statistics, Qinghai Nationalities University, }\\
{\small  Xining, Qinghai 810007, P.R.~China} \\
{\small $^{b}$ Qinghai Institute of Applied Mathematics,   Xining, Qinghai, 810007, P.R.~China} \\
{\small $^{c}$ School of Mathematics and System Sciences, Guangdong Polytechnic Normal University, }\\
{\small Guangzhou 510665, China }\\
{\small $^{d}$ Department of Mathematics, West Virginia University, Morgantown, WV, USA }}
\date{}

\maketitle
\noindent {\bf Abstract:} Given a graph or a digraph, let $M$ be  a combinatorial matrix (adjacency matrix and (signless) Laplacian matrix) of the (di)graph. The polynomial $\Phi_{k}^{M}(M,x) =\Phi_{(k,1^{n-k})}^{M}(M,x) = d_{(k,1^{n-k})}(xI-M)$ is called the hook immanantal polynomial of the (di)graph. This article establishes two identities: one is between the hook immanantal polynomial of a digraph and its derivatives, and the hook immanantal polynomials of all subgraphs obtained by deleting one arc from the digraph; the other is between the hook immanantal polynomial of a graph and its derivatives, and the hook immanantal polynomial of all subgraphs obtained by deleting one edge and its endpoints from the graph.
As applications, we show that the hook immanantal polynomials of combinatorial matrices of digraphs can be reconstructed from the  hook immanantal polynomials of combinatorial matrices of their subdigraphs. We also show that the hook immanantal polynomial of adjacency matirx of a graph can be reconstructed from the  hook immanantal polynomials of adjacency matrices of subgraphs of the graph.\\

%\smallskip
\noindent {\bf Keywords:}  Hook immanant; Graph matrix; Ulam-Kelly conjecture; Edge reconstruction; Vertex reconstruction \\
\noindent {\bf AMS subject classifications:} 05C20; 05C31; 05C60; 05E05; 15A15
\section{Introduction}
Let $G = (V(G), E(G))$ be a simple graph with vertex set $V(G) = \{v_{1}, v_{2}, \ldots, v_{n}\}$ and edge set $E(G) = \{e_{1}, e_{2}, \ldots, e_{m}\}$. The number of vertices and edges of $G$ denoted by $|V(G)|=n$ and $|E(G)|=m$, respectively. Let $e=v_sv_t$ be an edge of $G$, we may obtain a graph on $m-1$ edges by deleting $v_sv_t$ from $G$ but leaving the vertices and the remaining eages intact. The resulting graph is denoted by $G-v_sv_t$. Similarly, if $\{v_s,v_t|v_sv_t\in E(G)\}$ are two vertices of $G$, we may obtain a graph on $n-2$ vertices by deleting from $G$ the vertices in $\{v_s,v_t\}$ together with all the edges incident with a vertex in $\{v_s,v_t\}$. The resulting graph is denoted by $G-v_s-v_t$.
The degree matrix of $G$ is denoted by $D(G)=diag(d_{1}, d_{2}, \ldots, d_{n})$, where $d_{i}$ is the degree of vertex $v_{i}$ in $G$. The adjacency matrix of $G$ is an $n$ by $n$ matrix, denoted $A(G)= (a_{ij})$ whose entry $a_{ij}$ is given by
$$
a_{ij} =
\begin{cases}
1, & \text{if $v_i$ and $v_j$ are adjacent},\\
0, & \text{otherwise}.
\end{cases}
$$
Then $L(G)=D(G)-A(G)$ and $Q(G)=D(G)+A(G)$ are the Laplacian and signless Laplacian matrices of graph $G$, respectively.

Let $\overrightarrow{G} = (V(\overrightarrow{G}), E(\overrightarrow{G}))$ be a digraph  with no loops or parallel arcs, where the vertex set $V(\overrightarrow{G}) = \{v_{1}, v_{2}, \ldots, v_{n}\}$ and the arc set $E(\overrightarrow{G}) = \{\overrightarrow{e_{1}}, \overrightarrow{e_{2}}, \ldots, \overrightarrow{e_{m}}\}$. The number of vertices and edges of $\overrightarrow{G}$ denoted by $|V(\overrightarrow{G})|=n$ and $|E(\overrightarrow{G})|=m$, respectively. If $\overrightarrow{e}$ is an arc of $\overrightarrow{G}$, we may obtain a graph on $m-1$ arcs by deleting $\overrightarrow{e}$ from $\overrightarrow{G}$ but leaving the vertices and the remaining arcs intact. The resulting graph is denoted by ${\overrightarrow {G}}-{\overrightarrow {e}}$.
The adjacency matrix of digraph $\overrightarrow{G}$ is denoted $A(\overrightarrow{G}) = (a_{ij})_{n \times n}$, where
$$
a_{ij} =
\begin{cases}
1, & \text{if } (v_i, v_j) \in E(\overrightarrow{G}), \\
0, & \text{otherwise}.
\end{cases}
$$
The vertex in-degree diagonal matrix of digraph $\overrightarrow{G}$ is denoted $D(\overrightarrow{G}) = diag(d^{-}(v_{1}), d^{-}(v_{2}), \ldots,\\
 d^{-}(v_{n}))$, where the $d^{-}(v_{i})$ is in-degree of vertex $v_{i}$. Then $L(\overrightarrow{G})=D(\overrightarrow{G})-A(\overrightarrow{G})$ and $Q(\overrightarrow{G})=D(\overrightarrow{G})+A(\overrightarrow{G})$ are the Laplacian and signless Laplacian matrices of $\overrightarrow{G}$, respectively.

Let $\chi_{\lambda}$ be an irreducible character  of the symmetric group $S_n$, indexed by a partition  $\lambda$ of $n$. For an $n$ by $n$ matrix $M = (m_{ij})$, the {\em immanant} afforded by $\chi_{\lambda}$ is defined as
\begin{eqnarray*}
d_{\lambda}(M)=\sum_{\sigma\in S_{n}}\chi_{\lambda}(\sigma)\prod_{i=1}^{n}m_{i\sigma(i)}.
\end{eqnarray*}
\noindent
For $\lambda=(k,1^{n-k})$, the  immanant $d_{(k,1^{n-k})}(M)$ is called {\em hook immanant} of $M$, denoted by $d_{k}(M)$. In particular, $d_{1}(M)={\rm det} M $, the {\em determinant} of $M$, and $d_{n}(M)={\rm per}M$, the {\em permanent} of $M$. B\"{u}rgisser\cite{1} showed that computing the hook immanant of $M$ is VNP-complete. The study of immanants of matrices is well elaborated in \cite{20,3,2,5,7,6,4} and the references cited therein. Furthermore, the immanants of graph matrices received a lot of attention, as seen in \cite{11,12,8,9,10}, among others.

Let $I_{n}$ be an $n$ by $n$ identity matrix. The {\em hook immanantal polynomial} $\Phi_{k}^{M}(M,x)$ of an $n$ by $n$ matrix $M$ associated with the character $\chi_{(k,1^{n-k})}$ is the hook immanant $d_{(k,1^{n-k})}(xI_n-M)$, denoted by
\begin{eqnarray}\label{equ1}
\Phi_{k}^{M}(M,x)=\Phi_{(k,1^{n-k})}^{M}(M,x)=d_{(k,1^{n-k})}(xI_n-M).
\end{eqnarray}
B\"{u}rgisser\cite{1} showed that computing the hook immanantal polynomial of $M$ is VNP-complete. Merris\cite{15} stated that almost all trees share a complete set of immanantal polynomials. Cash\cite{17} generalized the Sachs theorem to immanantal polynomial and applied it to chemical graphs containing hydrocarbon structures. Yu\cite{16} represented the specific expression of the immanantal polynomial of $M$ using the basic subgraphs. Further study on immanantal polynomial is presented in \cite{19,18,34}.

By (\ref{equ1}), let M denote the adjacency matrix $A(\overrightarrow{G})$(or $A(G)$), the Laplacian matrix $L(\overrightarrow{G})$(or $L(G)$) and the signless Laplacian matrix $Q(\overrightarrow{G})$(or $Q(G)$) of a digraph $\overrightarrow{G}$(or a graph $G$), then $\Phi^{A}_{k}(\overrightarrow{G},x)$, $\Phi^{L}_{k}(\overrightarrow{G},x)$, $\Phi^{Q}_{k}(\overrightarrow{G},x)$, $\Phi^{A}_{k}(G,x)$, $\Phi^{L}_{k}(G,x)$ and $\Phi^{Q}_{k}(G,x)$ denote the hook immanantal polynomials of these matrices, respectively, as follow:
\begin{equation}\label{equ1.1}
\begin{split}
\begin{aligned}
\Phi^{A}_{k}(\overrightarrow{G}, x) &= d_{k} (xI_{n} - A(\overrightarrow{G})),~~ \\
\Phi^{L}_{k}(\overrightarrow{G}, x) &= d_{k} (xI_{n} - L(\overrightarrow{G})),~~ \\
\Phi^{Q}_{k}(\overrightarrow{G}, x) &= d_{k} (xI_{n} - Q(\overrightarrow{G})),~~
\end{aligned}
&
\begin{aligned}
\Phi^{A}_{k}(G, x) &= d_{k} (xI_{n} - A(G)), \\
\Phi^{L}_{k}(G, x) &= d_{k} (xI_{n} - L(G)), \\
\Phi^{Q}_{k}(G, x) &= d_{k} (xI_{n} - Q(G)).
\end{aligned}
\end{split}
\end{equation}
In particular, $\Phi_{1}^{A}(\overrightarrow{G},x)$($\Phi_{1}^{A}(G,x)$, resp.) and $\Phi_{n}^{A}(\overrightarrow{G},x)$($\Phi_{n}^{A}(G,x)$, resp.) are respectively the characteristic and permanental polynomial of $\overrightarrow{G}$($G$, resp.). For each function $f(x)$ in (\ref{equ1.1}), let $f'(x)$ denote its derivative.

For a digraph $\overrightarrow{G}$, we will show the following identity (\ref{equ2}) revealing a relationship among the hook immanantal polynomial of $\overrightarrow{G}$ and its derivatives, and the hook immanantal polynomials of all subdigraphs in $\{\overrightarrow{G}-\overrightarrow{e}|{\overrightarrow e}\in E({\overrightarrow G})\}$.
\begin{Theorem}\label{thm1}
Let $\overrightarrow{G}$ be a digraph with no loops or parallel arcs, which vertex set $V(\overrightarrow{G})=\{v_{1}, v_{2}, \ldots, v_{n}\}$ and arc set $E(\overrightarrow{G})= \{\overrightarrow{e_{1}}, \overrightarrow{e_{2}}, \ldots, \overrightarrow{e_{m}}\}$. For $M\in \{A,L,Q\}$, let $ \Phi^{M}_k(\overrightarrow{G}, x)$ be defined as in (\ref{equ1.1}) and $\Phi'^{M}_k(\overrightarrow{G}, x)$ denote its derivative. Then the following holds:
\begin{eqnarray}\label{equ2}
(m-n)\Phi^M_k({\overrightarrow G},x)+x\Phi'^M_k({\overrightarrow G},x)=\sum_{{\overrightarrow e}\in E({\overrightarrow G})}\Phi^M_k({\overrightarrow G}-{\overrightarrow e},x).
\end{eqnarray}
\end{Theorem}
For convenience, let ${\Phi}^{{A}({\square\!\square})}_k(G-v_s-v_t,x) = d_{(k-2,1^{n-k})}(xI_{n-2}-A(G-v_s-v_t))$ denote the hook immanantal polynomial of $A(G-v_s-v_t)$ corresponding to the partition $(k-2,1^{n-k})$, which is obtained from the partition $(k,1^{n-k})$ of $d_{k}(xI_{n}-A(G))$ by removing a rim hook of the form ${\square\!\square}$ from the Young tableau associated with $(k,1^{n-k})$, and let ${\Phi}^{{A}{(\scriptsize\rotatebox{90}{$\square\!\square$})}}_k(G-v_s-v_t,x) = d_{(k,1^{n-k-2})}(xI_{n-2}-A(G-v_s-v_t))$ denote the hook immanantal polynomial of $A(G-v_s-v_t)$ corresponding to the partition $(k,1^{n-k-2})$, which is obtained from the partition $(k,1^{n-k})$ of $d_{k}(xI_{n}-A(G))$ by removing a rim hook of the form \rotatebox{90}{$\square\!\square$} from the Young tableau associated with $(k,1^{n-k})$. Similarly to the digraph $\overrightarrow{G}$, there exists an identity for the hook polynomial of the graph $G$.

\begin{Theorem}\label{thm2} Let $G$ be a simple graph with vertex set $V(G)=\{v_{1}, v_{2}, \ldots, v_{n}\}$ and edge set $E(G)= \{e_{1}, e_{2}, \ldots, e_{m}\}$. For the graph polynomial $\Phi^{A}_k(G, x)$ defined as in (\ref{equ1.1}) and $\Phi'^{A}_k(G, x)$ denote its derivative. Then the following holds:
\begin{eqnarray}\label{equ3}
&&(m-n)\Phi^{A}_k(G,x)+x\Phi'^{A}_k(G,x)\notag\\
&=&\sum_{{v_sv_t}\in E(G)}[\Phi^{A}_k(G-v_sv_t,x)-{\Phi}^{{A}({\square\!\square})}_k(G-v_s-v_t,x)+{\Phi}^{{A}{(\scriptsize\rotatebox{90}{$\square\!\square$})}}_k(G-v_s-v_t,x)].
\end{eqnarray}
\end{Theorem}

The famous Ulam-Kelly's vertex reconstruction conjecture\cite{21} stated that each simple graph $G$ with at least three vertices can be uniquely reconstructed from its set of vertex induced subgraphs $\{G-v|v\in V(G)\}$. Harary\cite{22,23} proposed a similar conjecture, called edge reconstruction conjecture, which asserts that every simple graph $G$ with edge set $E(G)$ can be reconstructed from its set of edge subgraphs $\{G-e|e\in E(G)\}$ if $E(G) \geq 4$. Although many results for these two conjectures have been obtained\cite{28,29,30,31,32}, they remain open. Cvetkovi\'{c}\cite{33} posed a related question: Can the characteristic polynomial $\Phi^A_1(G,x)$ of a simple graph $G$ with vertex set $V(G)$ be reconstructed from the set of characteristic polynomials $\{\Phi^A_1(G - v,x)|v\in V(G)\}$ when $V(G) \geq 3$? Schwenk\cite{24} also independently raised the same question. Gutman and Cvetkovi\'{c}\cite{25} obtained some related results for this problem. Zhang et al.\cite{26,27} proved that if $|V(\overrightarrow{G})|\neq |E(\overrightarrow{G})|$, then the Laplacian (resp. signless Laplacian) characteristic polynomial and the Laplacian (resp. signless Laplacian) permanental polynomial of digraph $\overrightarrow{G}$ can be reconstructed from the characteristic polynomial (resp. permanental polynomial) of all subgraphs in $\{{\overrightarrow e}\in E({\overrightarrow G})\}$. They also proved that similar results hold for the characteristic polynomial (resp. permanental polynomial) of a simple graph $G$. The characteristic polynomial and permanental polynomial are special cases of hook immanantal polynomials. It naturally arises whether the same conclusion also holds for the hook immanantal polynomial of (di)graph that can be reconstructed from the hook immanantal polynomials of its subgraphs.

In the present study, we demonstrate that the hook immanantal polynomial of a digraph can be reconstructed from the hook immanantal polynomials of all subdigraphs in $\{\overrightarrow{G}-\overrightarrow{e}|{\overrightarrow e}\in E({\overrightarrow G})\}$. By taking $x=0$ in (\ref{equ2}), we note that if $m\neq n$, then
\begin{eqnarray*}
\Phi^M_k({\overrightarrow G},0)=\frac{1}{m-n}\sum_{{\overrightarrow e}\in E({\overrightarrow G})}\Phi^M_k({\overrightarrow G}-{\overrightarrow e},0)\;\;(M\in \{A,L,Q\}).
\end{eqnarray*}
Thus the above equation has a unique solution. This implies the following result.
\begin{Theorem}\label{thm3}
Let $\overrightarrow{G}$ be a digraph with no loops or parallel arcs, which vertex set $V(\overrightarrow{G})=\{v_{1}, v_{2}, \ldots, v_{n}\}$ and arc set $E(\overrightarrow{G})= \{\overrightarrow{e_{1}}, \overrightarrow{e_{2}}, \ldots, \overrightarrow{e_{m}}\}$. If $m\neq n$, then $\Phi^{M}_k(\overrightarrow{G},x)$ can be reconstructed from $\{\Phi^{M}_k(\overrightarrow{G}-\overrightarrow{e},x)|\overrightarrow{e}\in E(\overrightarrow{G})\}$ for $M\in \{A,L,Q\}$.
\end{Theorem}

Similarly, the same holds for graph $G$. By taking $x=0$ in (\ref{equ3}), we note that if $m\neq n$, then
\begin{eqnarray*}
\Phi^{A}_{k}(G,0)=\frac{1}{m-n}\sum_{{v_sv_t}\in E(G)}[\Phi_{k}^{A}(G-v_sv_t,0)-{\Phi}^{{A}({\square\!\square})}_k(G-v_s-v_t,0)+{\Phi}_{k}^{{A}(\scriptsize\rotatebox{90}{$\square\!\square$})}(G-v_s-v_t,0)].
\end{eqnarray*}
Thus the above equation has a unique solution. This imply the following result.
\begin{Theorem}\label{thm4}  Let $G$ be a simple graph with vertex set $V(G)=\{v_{1}, v_{2}, \ldots, v_{n}\}$ and edge set $E(G)= \{e_{1}, e_{2}, \ldots, e_{m}\}$. If $m\neq n$, then $\Phi^{A}_{k}(G, x)$ can be reconstructed from$\{\Phi^{A}_{k}(G-v_sv_t,x)|v_sv_t\in E(G)\}\cup\{{\Phi}_k^{A({\square\!\square})}(G-v_s-v_t,x),{\Phi}_k^{A({\scriptsize\rotatebox{90}{$\square\!\square$}})}(G-v_s-v_t,x)|v_{s}v_{t}\in E(G)\}$.
\end{Theorem}

The rest of this article is organized as follows. In Section \ref{sec2}, We  characterize some properties of hook immanant and hook immanantal polynomial. Base on these properties, we give the proof of Theorem \ref{thm1}. Similar to Section  \ref{sec2},  we give the proof of Theorem \ref{thm2} in section 3. Finally, We have made a brief summary about the hook immanantal polynomial reconstruction.
\section{The proof of theorem \ref{thm1}}\label{sec2}
Before we proceed to prove Theorem \ref{thm1}, We introduce the relevant definitions and lemmas, which are essential for the subsequent proofs.

Let $B$ be an $n \times n$ matrix, and let $B[i_1, \ldots, i_r]$ be a principal submatrix of $B$ that includes the rows and columns indexed by $\{i_1, \ldots, i_r\}\subset N_n=\{1, 2, \ldots, n\}$, and $B(i_1, \ldots, i_r)$ denotes the complement of this principal submatrix within $B$. For integer $1 \leq k \leq n$, let $Q_{k,n}$ denote the set of all strictly increasing integer sequences $\{i_1, \ldots, i_k\}$ selected from the set $N_n$. Consider two subsets $S = \{s_1, \ldots, s_i\} \subset N_n$ and $T = \{t_1, \ldots, t_j\} \subset N_n$. The submatrix $B_{T}^{S}$ is obtained by deleting all rows indexed by $S$ and all columns indexed by $T$ from matrix $B$. For simplicity, we denotes $B_{\{t_1, \ldots, t_j\}}^{\{s_1, \ldots, s_i\}}$ as $B_{t_1, \ldots, t_j}^{s_1, \ldots, s_i}$.
Define a matrix $B_{ij}=(b_{st}^{ij})_{n\times n}$, where $1\leq i,j\leq n$, and the elements of $B_{ij}$ are given by
$$
b_{st}^{ij}=
  \begin{cases}
   b_{st}, & \text{if $(s,t)\neq(i,j)$},\\
   0, & \text{otherwise}.
  \end{cases}
$$
For example, if
$$
B=
\left(
  \begin{array}{cccc}
    b_{11} & b_{12} &  b_{13} &  b_{14} \\
    b_{21} & b_{22} &  b_{23} &  b_{24} \\
    b_{31} & b_{32} &  b_{33} &  b_{34} \\
    b_{41} & b_{42} &  b_{43} &  b_{44} \\
  \end{array}
\right),
$$
then
$$
B_{11}=
\left(
  \begin{array}{cccc}
    0 & b_{12} &  b_{13} &  b_{14} \\
    b_{21} & b_{22} &  b_{23} &  b_{24} \\
    b_{31} & b_{32} &  b_{33} &  b_{34} \\
    b_{41} & b_{42} &  b_{43} &  b_{44} \\
  \end{array}
\right),~
B_{23}=
\left(
  \begin{array}{cccc}
    b_{11} & b_{12} &  b_{13} &  b_{14} \\
    b_{21} & b_{22} &  0 &  b_{24} \\
    b_{31} & b_{32} &  b_{33} &  b_{34} \\
    b_{41} & b_{42} &  b_{43} &  b_{44} \\
  \end{array}
\right).
$$

\begin{lemma}\label{lem2.1}{\rm (\cite{26})} Let $B = (b_{st})_{n \times n}$ be an $n$ by $n$ matrix over the complex field, and $B_{ij}$ is defined as above. Then the determinant of $B$ satisfies
\begin{eqnarray*}
(n^2 - n)\det(B) = \sum_{1 \leq i,j \leq n} \det(B_{ij}).
\end{eqnarray*}
\end{lemma}

\begin{lemma}\label{lem2.2}{\rm (\cite{35})} Let $B$ be any $n$ by $n$ matrix, then the hook immanant of matrix $B$ satisfies the following recursive formula
\begin{eqnarray*}
d_k(B)=\sum_{\alpha\in Q_{k-1,n}}{\rm per} B[\alpha]\det B(\alpha)-d_{k-1}(B).
\end{eqnarray*}
\end{lemma}

\begin{lemma}\label{lem2.3} Let $B=(b_{st})_{n\times n}$ be a square matrix of order $n$ over the complex field, and $B_{ij}$ is defined as above. Then the hook immanant of $B$ satisfies
\begin{eqnarray}\label{equ4}
(n^{2}-n)d_{k}(B)=\sum_{1\leq i,j\leq n}d_{k}(B_{ij}).
\end{eqnarray}
\end{lemma}
\begin{proof}
 We proceed by induction on $k$. By Lemma \ref{lem2.1}, if $k=1$, then $(n^2-n)\det(B) = \sum\limits_{1 \leq i,j \leq n} \det(B_{ij})$, and so (\ref{equ4}) holds. Assuming that $k\geq2$ and (\ref{equ4}) holds smaller values. By Lemma \ref{lem2.2}, $d_k(B)=\sum\limits_{\alpha\in Q_{k-1,n}}\text{per} B[\alpha]\det B(\alpha)-d_{k-1}(B)$. By the inductive hypothesis, we have $(n^{2}-n)d_{k-1}(B)=\sum\limits_{1\leq i,j\leq n}d_{k-1}(B_{ij})$. Using these two equalities, it suffices to prove
\begin{eqnarray}\label{equ5}
(n^{2}-n)\sum_{\alpha\in Q_{k-1,n}}\text{per} B[\alpha]\det B(\alpha)=\sum_{1\leq i,j\leq n}\sum_{\alpha\in Q_{k-1,n}}\text{per} B_{ij}[\alpha]\det B_{ij}(\alpha).
\end{eqnarray}
By the difinitions of $B_{ij}$, $B_{ij}[\alpha]$ and $B_{ij}(\alpha)$, for $\alpha \in Q_{k-1,n}$, $b_{ij}=0$ may be in $B_{ij}[\alpha]$ or $B_{ij}(\alpha)$, or not. For simplicity, we denote $N_n\setminus\alpha$ as the set difference between $N_n$ and the set $\alpha$. The right-hand side of (\ref{equ5}) is
\begin{eqnarray*}
&&\sum_{1\leq i,j\leq n}\sum_{\alpha\in Q_{k-1,n}}\text{per} B_{ij}[\alpha]\det B_{ij}(\alpha)\notag \\
&=&\sum_{\alpha\in Q_{k-1,n}}\sum_{i,j\in N_n}\text{per} B_{ij}[\alpha]\det B_{ij}(\alpha)\notag \\
&=&\sum_{\alpha\in Q_{k-1,n}}\sum_{i,j\in \alpha}\text{per} B_{ij}[\alpha]\det B_{ij}(\alpha)+\sum_{\alpha\in Q_{k-1,n}}\sum_{i,j\in N_n \setminus \alpha}\text{per} B_{ij}[\alpha]\det B_{ij}(\alpha)\notag \\
&&+\sum_{\alpha\in Q_{k-1,n}}\sum_{\substack{i\in \alpha\\j\in N_n \setminus \alpha}}\text{per} B_{ij}[\alpha]\det B_{ij}(\alpha)+\sum_{\alpha\in Q_{k-1,n}}\sum_{\substack{j\in \alpha\\i\in N_n \setminus \alpha}}\text{per} B_{ij}[\alpha]\det B_{ij}(\alpha)\notag \\
&=&\sum_{\alpha\in Q_{k-1,n}}\sum_{i,j\in \alpha}\text{per} B_{ij}[\alpha]\det B(\alpha)+\sum_{\alpha\in Q_{k-1,n}}\sum_{i,j\in N_n \setminus \alpha}\text{per} B[\alpha]\det B_{ij}(\alpha)\notag \\
&&+\sum_{\alpha\in Q_{k-1,n}}\sum_{\substack{i\in \alpha\\j\in N_n \setminus \alpha}}\text{per} B[\alpha]\det B(\alpha)+\sum_{\alpha\in Q_{k-1,n}}\sum_{\substack{j\in \alpha\\i\in N_n \setminus \alpha}}\text{per} B[\alpha]\det B(\alpha)\notag \\
&=&\sum_{\alpha\in Q_{k-1,n}}\sum_{i,j\in \alpha}(\text{per} B[\alpha]-b_{ij}\text{per}((B[\alpha])^{i}_{j})\det B(\alpha)\notag \\
&&+\sum_{\alpha\in Q_{k-1,n}}\sum_{i,j\in N_n \setminus \alpha}\text{per} B[\alpha](\det B(\alpha)-(-1)^{i+j}b_{ij}\det((B(\alpha))^{i}_{j})\notag \\
&&+\sum_{\alpha\in Q_{k-1,n}}\sum_{\substack{i\in \alpha\\j\in N_n \setminus \alpha}}\text{per} B[\alpha]\det B(\alpha)+\sum_{\alpha\in Q_{k-1,n}}\sum_{\substack{j\in \alpha\\i\in N_n \setminus \alpha}}\text{per} B[\alpha]\det B(\alpha)\notag \\
&=&\sum_{\alpha\in Q_{k-1,n}}\sum_{i,j\in \alpha}\text{per} B[\alpha]\det B(\alpha)+\sum_{\alpha\in Q_{k-1,n}}\sum_{i,j\in N_n \setminus \alpha}\text{per} B[\alpha]\det B(\alpha)\notag \\
&&+\sum_{\alpha\in Q_{k-1,n}}\sum_{\substack{i\in \alpha\\j\in N_n \setminus \alpha}}\text{per} B[\alpha]\det B(\alpha)+\sum_{\alpha\in Q_{k-1,n}}\sum_{\substack{j\in \alpha\\i\in N_n \setminus \alpha}}\text{per} B[\alpha]\det B(\alpha)\notag \\
&&-\sum_{\alpha\in Q_{k-1,n}}\sum_{i,j\in \alpha}b_{ij}\text{per}((B[\alpha])^{i}_{j})\det B(\alpha)\notag \\
&&-\sum_{\alpha\in Q_{k-1,n}}\sum_{i,j\in N_n\setminus\alpha}\text{per} B[\alpha](-1)^{i+j}b_{ij}\det((B(\alpha))^{i}_{j})\notag \\
&=&\sum_{\alpha\in Q_{k-1,n}}\sum_{i,j\in N_n}\text{per} B[\alpha]\det B(\alpha)\notag \\
&&-(k-1)\sum_{\alpha\in Q_{k-1,n}}\text{per} B[\alpha]\det B(\alpha)-(n-k+1)\sum_{\alpha\in Q_{k-1,n}}\text{per} B[\alpha]\det B(\alpha)\notag \\
&=&(n^2-n)\sum_{\alpha\in Q_{k-1,n}}\text{per} B[\alpha]\det B(\alpha).
\end{eqnarray*}
Lemma \ref{lem2.3} is proven.
\end{proof}

Let $m$ be the number of non-zero entries of $B$. If $b_{ij}=0$, then $B_{ij}=B$. Then the following result is equivalent to the lemma above.

\begin{lemma}\label{lem2.4} Let $B=(b_{st})_{n\times n}$ be a square matrix of order $n$ over the complex field, and $m$ be the number of non-zero entries of $B$. Then
\begin{eqnarray*}
(m-n)d_{k}(B)=\sum_{(i,j)\in I}d_{k}(B_{ij}).
\end{eqnarray*}
where $I=\{(i,j)|b_{ij}\neq0,~1\leq i, j\leq n\}$.
\end{lemma}

\begin{lemma}\label{lem2.5} Let $B=(b_{st})_{n\times n}$ be a square matrix of order $n$ over the complex field, and $\Phi_k^B(B,x)$ be the hook immanantal polynomial of $B$. Then
\begin{eqnarray}\label{equ10}
\Phi_k'^B(B,x)=\sum_{i=1}^{n}[d_{(k,1^{n-k-1})}((xI_n-B)_i^i)+d_{(k-1,1^{n-k})}((xI_n-B)_i^i)].
\end{eqnarray}
\end{lemma}
\begin{proof}
 We proceed by induction on $k$. By the derivative of the characteristic polynomial, if $k=1$, then $\Phi_1'^B(B,x)= \sum\limits_{i=1}\limits^{n}{\det ((xI_n-B)_i^i)}$, and so (\ref{equ10}) holds. Assuming that $k\geq2$ and (\ref{equ10}) holds smaller values. By Lemma \ref{lem2.2},  $\Phi_k^B(B,x)=\sum\limits_{\alpha\in Q_{k-1,n}}\text{per}(xI_n-B)[\alpha]\text{det}(xI_n-B)(\alpha)-\Phi^B_{k-1}(B,x)$. By the inductive hypothesis, we have $\Phi'^B_{k-1}(B,x)=\sum\limits_{i=1}\limits^{n}d_{(k-1,1^{n-k})}((xI_n-B)_i^i)+d_{(k-2,1^{n-k+1})}((xI_n-B)_i^i)$. Using these two equalities, it suffices to prove
\begin{eqnarray*}
&&\sum_{\alpha\in Q_{k-1,n}}[\text{per}(xI_n-B)[\alpha]\text{det}(xI_n-B)(\alpha)]'\\
&=&\sum_{i=1}^{n}\sum_{\alpha_1\in Q_{k-1,n-1}}\text{per}[(xI_n-B)_i^i][\alpha_1]\text{det}[(xI_n-B)_i^i](\alpha_1)\\
&&+\sum_{i=1}^{n}\sum_{\alpha_2\in Q_{k-2,n-1}}\text{per}[(xI_n-B)_i^i][\alpha_2]\text{det}[(xI_n-B)_i^i](\alpha_2).
\end{eqnarray*}
Note that
\begin{eqnarray*}
&(\text{per}(xI_n-B)[\alpha])'=\sum\limits_{i=1}\limits^{k-1}((\text{per}(xI_n-B)[\alpha])_i^i),\\
&(\text{det}(xI_n-B)(\alpha))'=\sum\limits_{i=1}\limits^{n-k+1}(\text{det}((xI_n-B)(\alpha))_i^i).
\end{eqnarray*}
Thus we have
\begin{eqnarray}\label{equ1.2}
&&\sum_{\alpha\in Q_{k-1,n}}[\text{per}(xI_n-B)[\alpha]\text{det}(xI_n-B)(\alpha)]'\notag\\
&=&\sum_{\alpha\in Q_{k-1,n}}(\text{per}(xI_n-B)[\alpha])'\text{det}(xI_n-B)(\alpha)\notag\\
&&+\sum_{\alpha\in Q_{k-1,n}}\text{per}(xI_n-B)[\alpha](\text{det}(xI_n-B)(\alpha))'\notag\\
&=&\sum_{i=1}^{k-1}\sum_{\alpha\in Q_{k-1,n}}((\text{per}(xI_n-B)[\alpha])_i^i)\text{det}(xI_n-B)(\alpha)\notag\\
&&+\sum_{i=1}^{n-k+1}\sum_{\alpha\in Q_{k-1,n}}\text{per}(xI_n-B)[\alpha](\text{det}((xI_n-B)(\alpha))_i^i).
\end{eqnarray}
The sum in equation (\ref{equ1.2}) includes all $(k-2)$-th order principal submatrices of $xI_n-B$ and their corresponding $(n-k+1)$-th order complementary submatrices. Additionally, it incorporates all $(k-1)$-th order principal submatrices of $xI_n-B$ along with the corresponding $(n-k)$-th order complementary submatrices. This process involves computing the products and sums of the characteristic polynomials and the permanental polynomials associated with these submatrices. Therefore, equation (\ref{equ1.2}) is equal to
\begin{eqnarray*}
&&\sum_{i=1}^{n}\sum_{\alpha_2\in Q_{k-2,n-1}}\text{per}((xI_n-B)_i^i)[\alpha_2]\text{det}((xI_n-B)_i^i)(\alpha_2)\\
&&+\sum_{i=1}^{n}\sum_{\alpha_1\in Q_{k-1,n-1}}\text{per}((xI_n-B)_i^i)[\alpha_1]\text{det}((xI_n-B)_i^i)(\alpha_1).\\
\end{eqnarray*}
Lemma \ref{lem2.5} is thus proven.
\end{proof}

Having set the necessary  preliminaries, we now present the demonstration of Theorem \ref{thm1}.

\textbf{Proof of Theorem \ref{thm1}.}
Let $\beta$ and $\gamma$ be real numbers satisfying $\gamma\neq0$. For convenience, we define $\Phi(\overrightarrow{G},x) = d_{k} (xI_{n} - \beta D(\overrightarrow{G}) - \gamma A(\overrightarrow{G}))$.
Note that $xI_{n}-\beta D(\overrightarrow{G})-\gamma A(\overrightarrow{G})$ has $m+n$ non-zero entries. By Lemma \ref{lem2.3} and Lemma \ref{lem2.4}, we know that
\begin{eqnarray}\label{equ11}
&&m \cdot d_{k}(xI_{n} - \beta D(\overrightarrow{G}) - \gamma A(\overrightarrow{G})) \notag\\
&=&\sum_{i=1}^{n} d_{k}(xI_{n}^{(i)} - \beta D(\overrightarrow{G})^{(i)} - \gamma A(\overrightarrow{G})) + \sum_{\overrightarrow{e} \in E(\overrightarrow{G})} d_{k}(xI_{n} - \beta D(\overrightarrow{G}) - \gamma A(\overrightarrow{G}-\overrightarrow{e})).
\end{eqnarray}
where $I_{n}^{(i)}$ denotes an $n$-order diagonal matrix with the $i$-th diagonal element being 0 and the rest being 1 and $D(\overrightarrow{G})^{(i)}=diag(d^{-}(v_{1}), \ldots, d^{-}(v_{i-1}), 0, d^{-}(v_{i+1}), \ldots, d^{-}(v_{n}))$. Without loss of generality, let $\overrightarrow{e}=v_{s}v_{t}$.
To facilitate the arguments let $\Delta_i = d_{k}(xI_{n}^{(i)} - \beta D(\overrightarrow{G})^{(i)} - \gamma A(\overrightarrow{G}))$ and $\Delta_{\overrightarrow{e}} = d_{k}(xI_{n} - \beta D(\overrightarrow{G}) - \gamma A(\overrightarrow{G}-\overrightarrow{e}))$. By the Murnaghan-Nakayama rule \cite{36}, we know that $\chi_{(k,1^{n-k})}(\sigma) = \chi_{(k,1^{n-k-1})}(\sigma_1) + \chi_{(k-1,1^{n-k})}(\sigma_1)$, where $\sigma_1$ is the new permutation obtained by deleting a fixed point from $\sigma$, and bringing it into the related immanant function. Then
\begin{eqnarray}\label{equ12}
\Delta_i&=&d_{(k,1^{n-k})}(xI_{n}-\beta D(\overrightarrow{G})-\gamma A(\overrightarrow{G}))-(x-\beta d^{-}(v_{i}))d_{(k,1^{n-k-1})}[(xI_{n}-\beta D(\overrightarrow{G})-\gamma A(\overrightarrow{G}))_i^i]\notag\\
&&-(x-\beta d^{-}(v_{i}))d_{(k-1,1^{n-k})}[(xI_{n}-\beta D(\overrightarrow{G})-\gamma A(\overrightarrow{G}))_i^i].
\end{eqnarray}
\begin{eqnarray}\label{equ13}
\Delta_{\overrightarrow{e}}&=&d_{k}(xI_{n}-\beta D(\overrightarrow{G}-{\overrightarrow{e}})-\gamma A(\overrightarrow{G}-\overrightarrow{e}))-\beta d^{-}(v_{t}) d_{(k,1^{n-k-1})}[(xI_{n}-\beta D(\overrightarrow{G})\notag\\
&&-\gamma A(\overrightarrow{G}))_t^t]-\beta d_{(k-1,1^{n-k})}[(xI_{n}-\beta D(\overrightarrow{G})-\gamma A(\overrightarrow{G}))_t^t].
\end{eqnarray}
By (\ref{equ12}) and (\ref{equ13}), we have
\begin{eqnarray}\label{equ14}
\sum_{i=1}^{n}\Delta_i&=&\sum_{i=1}^{n}d_{(k,1^{n-k})}(xI_{n}-\beta D(\overrightarrow{G})-\gamma A(\overrightarrow{G}))-\sum_{i=1}^{n}(x-\beta d^{-}(v_{i}))d_{(k,1^{n-k-1})}[(xI_{n}\notag\\
&&-\beta D(\overrightarrow{G})-\gamma A(\overrightarrow{G}))_i^i]-\sum_{i=1}^{n}(x-\beta d^{-}(v_{i}))d_{(k-1,1^{n-k})}[(xI_{n}-\beta D(\overrightarrow{G})-\gamma A(\overrightarrow{G}))_i^i].\notag\\
&=&n\Phi(\overrightarrow{G},x)-x\sum_{i=1}^{n}d_{(k,1^{n-k-1})}[(xI_{n}-\beta D(\overrightarrow{G})-\gamma A(\overrightarrow{G}))_i^i]-x\sum_{i=1}^{n}d_{(k-1,1^{n-k})}[(xI_{n}\notag\\
&&-\beta D(\overrightarrow{G})-\gamma A(\overrightarrow{G})_i^i]+\beta\sum_{i=1}^{n} d^{-}(v_{i})d_{(k,1^{n-k-1})}[(xI_{n}-\beta D(\overrightarrow{G})-\gamma A(\overrightarrow{G}))_i^i]\notag\\
&&+\beta\sum_{i=1}^{n}d^{-}(v_{i})d_{(k-1,1^{n-k})}[(xI_{n}-\beta D(\overrightarrow{G})-\gamma A(\overrightarrow{G}))_i^i].\\
\sum_{\overrightarrow{e}\in E(\overrightarrow{G})}\Delta_{\overrightarrow{e}}&=&\sum_{\overrightarrow{e}\in E(\overrightarrow{G})}d_{k}(xI_{n}-\beta D(\overrightarrow{G}-\overrightarrow{e})-\gamma A(\overrightarrow{G}-\overrightarrow{e}))-\beta\sum_{v_{s}v_{t}\in E(\overrightarrow{G})}d^{-}(v_{t}) d_{(k,1^{n-k-1})}[(xI_{n}\notag\\
&&-\beta D(\overrightarrow{G})-\gamma A(\overrightarrow{G}))_t^t]-\beta\sum_{v_{s}v_{t}\in E(\overrightarrow{G})}d^{-}(v_{t})d_{(k-1,1^{n-k})}[(xI_{n}-\beta D(\overrightarrow{G})-\gamma A(\overrightarrow{G}))_t^t]\notag\\
&=&\sum_{\overrightarrow{e}\in E(\overrightarrow{G})}\Phi(\overrightarrow{G}-\overrightarrow{e},x)-\beta\sum_{t=1}^{n} d^{-}(v_{t})d_{(k,1^{n-k-1})}[(xI_{n}-\beta D(\overrightarrow{G})-\gamma A(\overrightarrow{G}))_t^t]\notag\\
&&-\beta\sum_{t=1}^{n} d^{-}(v_{t})d_{(k-1,1^{n-k})}[(xI_{n}-\beta D(\overrightarrow{G})-\gamma A(\overrightarrow{G}))_t^t].
\end{eqnarray}
By Lemma \ref{lem2.5}, we have
\begin{eqnarray}\label{equ16}
&&\Phi'(\overrightarrow{G},x)\notag\\
&=&\sum_{i=1}^{n}d_{(k,1^{n-k-1})}[(xI_{n}-\beta D(\overrightarrow{G})-\gamma A(\overrightarrow{G}))_i^i]+d_{(k-1,1^{n-k})}[(xI_{n}-\beta D(\overrightarrow{G})-\gamma A(\overrightarrow{G}))_i^i].~~~~~~~~~~~~~
\end{eqnarray}
By (\ref{equ11}) through (\ref{equ16}), we obtain
\begin{eqnarray*}
m\Phi(\overrightarrow{G},x)=n\Phi(\overrightarrow{G},x)-x\Phi'(\overrightarrow{G},x)+\sum_{\overrightarrow{e}\in E(\overrightarrow{G})}\Phi({\overrightarrow{G}}-\overrightarrow{e},x).
\end{eqnarray*}
Theorem \ref{thm1} is proven.\qed

By Theorem \ref{thm1}, we note that if $m\neq n$, then the following equation holds
\begin{eqnarray*}
\Phi^M_k({\overrightarrow G},0)=\frac{1}{m-n}\sum_{{\overrightarrow e}\in E({\overrightarrow G})}\Phi^M_k({\overrightarrow G}-{\overrightarrow e},0)\;\;(M\in \{A,L,Q\}).
\end{eqnarray*}
Thus the above equation has a unique solution. This implies that Theorem \ref{thm3} holds. Note that the characteristic polynomial and permanental polynomial are special cases of hook immanantal polynomials, whereby we can obtain the following corollaries.

\begin{cor}{\rm (\cite{26})} Given $\overrightarrow{G}$ a digraph with no loops or parallel arcs, with vertex set $V(\overrightarrow{G})=\{v_{1}, v_{2},\ldots,v_{n}\}$ and edge set $E(\overrightarrow{G})=\{\overrightarrow{e}_{1}, \overrightarrow{e}_{2},\ldots,\overrightarrow{e}_{m}\}$. If $m\neq n$, then $\Phi_{n}^{M}(\overrightarrow{G},x)$ can be reconstructed from $\{\Phi_{n}^{M}(\overrightarrow{G}-\overrightarrow{e},x)|\overrightarrow{e}\in E(\overrightarrow{G})\}$$(M\in\{A,L,Q\})$.
\end{cor}

\begin{cor}{\rm (\cite{26})} Given $\overrightarrow{G}$ a digraph with no loops or parallel arcs, with vertex set $V(\overrightarrow{G})=\{v_{1}, v_{2},\ldots,v_{n}\}$ and edge set $E(\overrightarrow{G})=\{\overrightarrow{e}_{1}, \overrightarrow{e}_{2},\ldots,\overrightarrow{e}_{m}\}$. If $m\neq n$, then $\Phi_{1}^{M}(\overrightarrow{G},x)$ can be reconstructed from $\{\Phi_{1}^{M}(\overrightarrow{G}-\overrightarrow{e},x)|\overrightarrow{e}\in E(\overrightarrow{G})\}$$(M\in\{A,L,Q\})$.
\end{cor}

\section{The proof of Theorem \ref{thm2}}\label{sec3}
Before we proceed to prove Theorem \ref{thm2}, We introduce the relevant definitions and lemmas, which are essential for the subsequent proofs.

Let $B=(b_{st})_{n\times n}$ be a complex $n$-order matrix. Define $B_{[ij]}={(b_{st}^{ij})}_{n\times n}$ as follows
\[
b_{st}^{ij}=
  \begin{cases}
   b_{st}, & \text{if $(s,t)\neq(i,j)$ and $(s,t)\neq(j,i)$},\\
   0, & \text{otherwise}.
  \end{cases}
\]
Obviously, when $b_{ij}=b_{ji}=0$, then $B=B_{[ij]}=B_{[ji]}$.
Define $B_{(ij)}={(b_{st}^{ij})}_{n\times n}$, specifically as
\[
b_{st}^{ij}=
  \begin{cases}
   b_{st}, & \text{if $(s,t)=(i,j)$},\\
   b_{st}, & \text{if $s\neq i$ and $t\neq j$},\\
   0, & \text{otherwise}.
  \end{cases}
\]
For example, if
$$
B=
\left(
  \begin{array}{cccc}
    b_{11} & b_{12} &  b_{13} &  b_{14} \\
    b_{21} & b_{22} &  b_{23} &  b_{24} \\
    b_{31} & b_{32} &  b_{33} &  b_{34} \\
    b_{41} & b_{42} &  b_{43} &  b_{44} \\
  \end{array}
\right),
$$
then
$$
B_{[11]}=
\left(
  \begin{array}{cccc}
    0 & b_{12} &  b_{13} &  b_{14} \\
    b_{21} & b_{22} &  b_{23} &  b_{24} \\
    b_{31} & b_{32} &  b_{33} &  b_{34} \\
    b_{41} & b_{42} &  b_{43} &  b_{44} \\
  \end{array}
\right),
B_{[23]}=
\left(
  \begin{array}{cccc}
    b_{11} & b_{12} &  b_{13} &  b_{14} \\
    b_{21} & b_{22} &  0 &  b_{24} \\
    b_{31} & 0 &  b_{33} &  b_{34} \\
    b_{41} & b_{42} &  b_{43} &  b_{44} \\
  \end{array}
\right).
$$
$$
B_{(11)}=
\left(
  \begin{array}{cccc}
    b_{11} & 0 &  0 &  0 \\
    0 & b_{22} &  b_{23} &  b_{24} \\
    0 & b_{32} &  b_{33} &  b_{34} \\
    0 & b_{42} &  b_{43} &  b_{44} \\
  \end{array}
\right),
B_{(23)}=
\left(
  \begin{array}{cccc}
    b_{11} & b_{12} &  0 &  b_{14} \\
    0 & 0 &  b_{23}&  0 \\
    b_{31} & b_{32} &  0 &  b_{34} \\
    b_{41} & b_{42} &  0 &  b_{44} \\
  \end{array}
\right).
$$

\begin{lemma}\label{lem3.1}Let $B=(b_{st})_{n\times n}$ be a square matrix of order $n$ over the complex field, and $B_{(ij)}$ be defined as above. Then the hook immanant of the matrix $B$ satisfies:
\begin{eqnarray}
\sum_{1\leq i\leq j \leq n}d_{k}(B_{(ij)})=nd_{k}(B)
\end{eqnarray}
\end{lemma}

\begin{proof}
 By the definition $d_{k}(B)=\sum\limits_{\sigma\in S_{n}}\chi_{k}(\sigma)\prod\limits_{i=1}\limits^{n}b_{i\sigma(i)}$, we know $d_{k}(B)=d_{k}(B_{(ij)})+d_{k}(B_{ij})$. Therefore,
$$\sum_{1\leq i\leq j \leq n}d_{k}(B)=\sum_{1\leq i\leq j \leq n}[d_{k}(B_{(ij)})+d_{k}(B_{ij})].$$ By Lemma \ref{lem2.4}, we then have$\sum\limits_{1\leq i\leq j \leq n}d_{k}(B_{(ij)})=nd_{k}(B)$.
\end{proof}

\begin{lemma}\label{lem3.2}Let $B=(b_{st})_{n\times n}$ be a square matrix of order $n$ over the complex field, and let $B_{[ij]}$ be defined as above. Then the hook immanant of matrix $B$ satisfies:
$$\frac{1}{2}(n^2-n)d_{k}(B)=\sum_{1\leq i\leq j \leq n}d_{k}(B_{[ij]})-\sum_{1\leq i< j \leq n}b_{ij}^2d_{(k-2,1^{n-k})}(B_{i,j}^{i,j})+\sum_{1\leq i< j \leq n}b_{ij}^2d_{(k,1^{n-k-2})}(B_{i,j}^{i,j}).$$
\end{lemma}

\begin{proof}
By the Murnaghan-Nakayama rule \cite{36}, we obtained that $\chi_{(k,1^{n-k})}(\sigma)=\chi_{(k-2,1^{n-k})}(\sigma_2)-\chi_{(k,1^{n-k-2})}(\sigma_2)$, where $\sigma_2$ is the new permutation obtained by deleting a transposition from $\sigma$. Substituting this into the corresponding immanant function and observing the structures of $B_{[ij]}$ and $B_{(ij)}$, using the principle of inclusion-exclusion, the following equations are obtained:
\begin{eqnarray*}
&&d_{k}(B_{[ii]})=d_{k}(B)-d_{k}(B_{(ii)}),\\
&&d_{k}(B_{[ij]})=d_{k}(B)-d_{k}(B_{(ij)})-d_{k}(B_{(ji)})+b_{ij}^2[d_{(k-2,1^{n-k})}(B_{i,j}^{i,j})-d_{(k,1^{n-k-2})}(B_{i,j}^{i,j})].
\end{eqnarray*}
Using the above equations and Lemma \ref{lem3.1}, we get that
\begin{eqnarray*}
&&\sum_{1\leq i\leq j \leq n}d_{k}(B_{[ij]})\\
&=&\sum_{i=1}^nd_{k}(B_{[ii]})+\sum_{1\leq i< j \leq n}d_{k}(B_{[ij]})\\
&=&\sum_{i=1}^n[d_{k}(B)-d_{k}(B_{(ii)})]+\sum_{1\leq i< j \leq n}[d_{k}(B)-d_{k}(B_{(ij)})-d_{k}(B_{(ji)})]\\
&&+\sum_{1\leq i< j \leq n}[b_{ij}^2d_{(k-2,1^{n-k})}(B_{i,j}^{i,j})-b_{ij}^2d_{(k,1^{n-k-2})}(B_{i,j}^{i,j})] \notag \\
&=&\sum_{i=1}^nd_{k}(B)+\sum_{1\leq i< j \leq n}d_{k}(B)-\sum_{i=1}^nd_{k}(B_{(ii)})-\sum_{1\leq i< j \leq n}[d_{k}(B_{(ij)})+d_{k}(B_{(ji)})]\\
&&+\sum_{1\leq i< j \leq n}[b_{ij}^2d_{(k-2,1^{n-k})}(B_{i,j}^{i,j})-b_{ij}^2d_{(k,1^{n-k-2})}(B_{i,j}^{i,j})] \\
&=&\frac{1}{2}(n^2+n)d_{k}(B)-\sum_{1\leq i\leq j \leq n}d_{k}(B_{(ij)})+\sum_{1\leq i< j \leq n}[b_{ij}^2d_{(k-2,1^{n-k})}(B_{i,j}^{i,j})-b_{ij}^2d_{(k,1^{n-k-2})}(B_{i,j}^{i,j})] \\
&=&\frac{1}{2}(n^2+n)d_{k}(B)-nd_{k}(B)+\sum_{1\leq i< j \leq n}[b_{ij}^2d_{(k-2,1^{n-k})}(B_{i,j}^{i,j})-b_{ij}^2d_{(k,1^{n-k-2})}(B_{i,j}^{i,j})]\\
&=&\frac{1}{2}(n^2-n)d_{k}(B)+\sum_{1\leq i< j \leq n}[b_{ij}^2d_{(k-2,1^{n-k})}(B_{i,j}^{i,j})-b_{ij}^2d_{(k,1^{n-k-2})}(B_{i,j}^{i,j})].
\end{eqnarray*}
Lemma \ref{lem3.2} is proven.
\end{proof}

Let the number of non-zero non-diagonal elements of $B$ be $2m$, and the number of zeros among the non-diagonal elements be $2k$. Then $2m=n^2-n-2k$. The following result is equivalent to Lemma \ref{lem3.2} when applied to symmetric matrix.

\begin{lemma}\label{lem3.3}
Let $B$ be a symmetric matrix of order $n$ over the complex field, and let the number of non-zero non-diagonal entries of $B$ be $2m$, and the number of diagonal entries that are zero be $c$. Then
\begin{eqnarray*}
(m-c)d_{k}(B)=\sum_{(i,j)\in I_1}d_{k}(B_{[ij]})-\sum_{(i,j)\in I_2}[b_{ij}^2d_{(k-2,1^{n-k})}(B_{i,j}^{i,j})+b_{ij}^2d_{(k,1^{n-k-2})}(B_{i,j}^{i,j})].
\end{eqnarray*}
where $I_1=\{(i,j)|b_{ij}\neq 0,1\leq i\leq j \leq n\}$ and $I_2=\{(i,j)|b_{ij}\neq 0,1\leq i< j \leq n\}$.
\end{lemma}

Let us now turn our attention to the proof of Theorem \ref{thm2}.

\textbf{ Proof of Theorem \ref{thm2}.}
Note that $G$ has $m$ edges, and the number of non-zero non-diagonal elements in $xI_n-A(G)$ is $2m$. For an adge $e\in E(G)$, let $e={v_sv_t}$. By Lemma \ref{lem3.3},
\begin{eqnarray}\label{equ18}
&&m\cdot d_{k}(xI_n-A(G))\notag\\
&=&\sum_{i=1}^nd_{k}[((xI_n)_{[ii]})-A(G)]+\sum_{{v_sv_t}\in E(G)}d_{k}(xI_n-A(G-{v_sv_t}))\notag\\
&&-\sum_{{v_s v_t}\in E(G)}[a_{st}^2d_{(k-2,1^{n-k})}((xI_n-A(G))_{s,t}^{s,t})+a_{st}^2d_{(k,1^{n-k-2})}((xI_n-A(G))_{s,t}^{s,t})].
\end{eqnarray}
By the Murnaghan-Nakayama rule \cite{36}, we have $\chi_{(k,1^{n-k})}(\sigma) = \chi_{(k-1,1^{n-k})}(\sigma_1)+\chi_{(k,1^{n-k-1})}(\sigma_1) $, where $\sigma_1$ is the new permutation obtained by deleting a fixed point from $\sigma$, and bringing it into the related immanant function. Then
\begin{eqnarray}\label{equ19}
&&d_{k}(((xI_n)_{[ii]})-A(G))\notag\\
&=&d_{k}(xI_n-A(G))-xd_{(k-1,1^{n-k})}((xI_n-A(G))_i^i)-xd_{(k,1^{n-k-1})}((xI_n-A(G))_i^i).
\end{eqnarray}
Combining (\ref{equ18}) and (\ref{equ19}), we obtain
\begin{eqnarray*}
&&m\cdot d_{k}(xI_n-A(G))\\
&=&\sum_{i=1}^nd_{k}[((xI_n)_{[ii]})-A(G)]+\sum_{{v_sv_t}\in E(G)}d_{k}(xI_n-A(G-{v_sv_t}))\\
&&-\sum_{{v_sv_t}\in E(G)}[a_{st}^2d_{(k-2,1^{n-k})}(A(G)_{s,t}^{s,t})+a_{st}^2d_{(k,1^{n-k-2})}(A(G)_{s,t}^{s,t})]\\
&=&\sum_{i=1}^nd_{k}(xI_n-A(G))-x\sum_{i=1}^n[d_{(k-1,1^{n-k})}((xI_n-A(G))_i^i)+d_{(k,1^{n-k-1})}((xI_n-A(G))_i^i)]\\
&&+\sum_{{v_sv_t}\in E(G)}\Phi'^{A}_{k}(G-{v_sv_t},x)-\sum_{{v_sv_t}\in E(G)}[a_{st}^2d_{(k-2,1^{n-k})}(A(G)_{s,t}^{s,t})+a_{st}^2d_{(k,1^{n-k-2})}(A(G)_{s,t}^{s,t})]\\
&=&n\Phi^{A}_{k}(G,x)-x\Phi'^{A}_{k}(G,x)+\sum_{{v_sv_t}\in E(G)}\Phi_{k}^{A}(G-{v_sv_t},x)\\
&&-\sum_{{v_sv_t}\in E(G)}[a_{st}^2d_{(k-2,1^{n-k})}(A(G)_{s,t}^{s,t})+a_{st}^2d_{(k,1^{n-k-2})}(A(G)_{s,t}^{s,t})]\\
&=&n\Phi^{A}_{k}(G,x)-x\Phi'^{A}_{k}(G,x)+\sum_{{v_sv_t}\in E(G)}\Phi^{A}_k(G-v_sv_t,x)\\
&&-\sum_{{v_sv_t}\in E(G)}[{\Phi}^{{A}({\square\!\square})}_k(G-v_s-v_t,x)+{\Phi}^{{A}{(\scriptsize\rotatebox{90}{$\square\!\square$})}}_k(G-v_s-v_t,x)].
\end{eqnarray*}\qed

By Theorem \ref{thm2}, we note that if $m\neq n$, then the following equation holds
\begin{eqnarray*}
\Phi^{A}_{k}(G,0)=\frac{1}{m-n}\sum_{{v_sv_t}\in E(G)}[\Phi_{k}^{A}(G-v_sv_t,0)-{\Phi}^{{A}({\square\!\square})}_k(G-v_s-v_t,0)+{\Phi}_{k}^{{A}(\scriptsize\rotatebox{90}{$\square\!\square$})}(G-v_s-v_t,0)].
\end{eqnarray*}
Thus the above equation has a unique solution. This implies that Theorem \ref{thm4} holds. Note that the characteristic polynomial and permanental polynomial are special cases of hook immanantal polynomials, whereby we can obtain the following corollaries.

\begin{cor}{\rm (\cite{27})}Let $G$ be a simple graph with vertex set $V(G)=\{v_{1}, v_{2}, \ldots, v_{n}\}$ and edge set $E(G)=\{e_{1}, e_{2}, \ldots, e_{m}\}$. If $m \neq n$, then $\Phi_{n}^{A}(G,x)$ is reconstructed from $\{\Phi_{n}^{A}(G-e,x) | e \in E(G)\} \cup \{\Phi_{n}^{{A}({\square\!\square})}(G-v_s-v_t,x)|v_sv_t \in E(G)\}$.
\end{cor}

\begin{cor}{\rm (\cite{27})}Let $G$ be a simple graph with vertex set $V(G)=\{v_{1}, v_{2}, \ldots, v_{n}\}$ and edge set $E(G)=\{e_{1}, e_{2}, \ldots, e_{m}\}$. If $m \neq n$, then $\Phi_{1}^{A}(G,x)$ is reconstructed from $\{\Phi_{1}^{A}(G-e,x) | e \in E(G)\} \cup \{\Phi_{1}^{{A}{(\scriptsize\rotatebox{90}{$\square\!\square$})}}(G-v_s-v_t,x)|v_sv_t \in E(G)\}$.
\end{cor}
\section{Summary}
In this article, we characterize two identities on hook immanantal polynomial of  combinatorial matrix(adjacency matrix, (signless) Laplacian matrix). As applications, we demonstrate that  the hook immanantal polynomial of  combinatorial matrix of a digraph can be  reconstructed. And we also emonstrate that  the hook immanantal polynomial of  adjacency matrix of a graph can be  reconstructed.  A natural problem: can be  reconstructed on the hook immanantal polynomial of  (signless) Laplacian matrix of a graph?
 Zhang et al.\cite{27} have proved that $\Phi_{1}^{L}(G,x)$ is reconstructed from $\{\Phi_{1}^{L}(G-e,x) | e \in E(G)\}$. However, if $k\neq1$, this property remains unknown for hook immanantal polynomials of a graph $G$. In the future, we can continue to explore the aforementioned unsolved questions.

\noindent{\bf Data Availability}\\
{No data were used to support this study.}

\end{document}